\documentclass{amsart}
\def\ack{\subsubsection*{Acknowledgement.}}%
\usepackage{amsmath,amsfonts,amsthm,amssymb,verbatim}
\newtheorem{theo}{Theorem}
\newtheorem{lemma}{Lemma}[section]
\newtheorem*{thma}{Theorem A}
\numberwithin{equation}{section}
\def \isnatural {\in\mathbb{N}}
\def \isreal {\in\mathbb{R}}
\def \iscomplex {\in\mathbb{C}}
\def \toinfty {\rightarrow\infty}
\newcommand{\tef}{transcendental entire function}
\newcommand\qfor{\quad\text{for }}
\begin{document}
%
%
\title[Simply connected fast escaping Fatou components]{Simply connected fast escaping Fatou components}
\author{D. J. Sixsmith}
\address{Department of Mathematics and Statistics \\
	 The Open University \\
   Walton Hall\\
   Milton Keynes MK7 6AA\\
   UK}
\email{david.sixsmith@open.ac.uk}
%
%
\begin{abstract}
We give an example of a {\tef} with a simply connected fast escaping Fatou component, but with no multiply connected Fatou components. We also give a new criterion for points to be in the fast escaping set.
\end{abstract}
\maketitle
%
%
%
\section{Introduction}
Suppose that $f:\mathbb{C}\rightarrow\mathbb{C}$ is a {\tef}. The \itshape Fatou set \normalfont $F(f)$ is defined as the set of points $z\iscomplex$ such that $(f^n)_{n\isnatural}$ is a normal family in a neighborhood of $z$. The \itshape Julia set \normalfont $J(f)$ is the complement of $F(f)$. An introduction to the properties of these sets was given in \cite{MR1216719}.

The \itshape fast escaping set \normalfont $A(f)$ was introduced in \cite{MR1684251}. We use the definition $$A(f) = \{z : \text{there exists } \ell \isnatural \text{ such that } |f^{n+\ell}(z)| \geq M^n(R, f), \text{ for } n \isnatural\},$$  given in \cite{fast}. Here, the \itshape maximum modulus function \normalfont $M(r, f) = \max_{|z|=r} |f(z)|,$ for $r > 0,$ $M^n(r, f)$ denotes repeated iteration of $M(r,f)$ with respect to $r$, and $R > 0$ can be taken to be any value such that $M(r, f) > r$, for $r \geq R$.  We write $M(r)$ when it is clear from the context which function is being considered.

Suppose that $U=U_0$ is a component of $F(f)$. If $U \cap A(f)\ne\emptyset$, then $\overline{U} \subset A(f)$ \cite[Theorem 1.2]{fast}. We call a Fatou component in $A(f)$ \itshape fast escaping\normalfont. Denote by $U_n$ the component of $F(f)$ containing $f^n(U)$. We say that a component is \itshape wandering \normalfont if $U_n = U_m$ implies that $n=m$. All fast escaping Fatou components are wandering; \cite[Lemma 4]{MR1684251} and \cite[Corollary 4.2]{fast}.

For a {\tef}, all multiply connected Fatou components are fast escaping; \cite[Theorem 2]{MR2117213} and \cite[Theorem 4.4]{fast}. The first example of a {\tef} with a multiply connected Fatou component was constructed by Baker in \cite{MR0153842}. Other examples are found in, for example, \cite{MR796748}, \cite{MRunknown}, \cite{2011arXiv1109.1794B} and \cite{MR2458806}.

The only known example of a simply connected fast escaping Fatou component was given by Bergweiler \cite{MRunknown}, using a quasi-conformal surgery technique from \cite{MR2458806}. This function also has multiply connected Fatou components. In fact, in \cite{MRunknown}, the properties of the multiply connected Fatou components are used to show that the simply connected Fatou components are fast escaping.

This prompts the question of whether a {\tef} can have simply connected fast escaping Fatou components without having multiply connected Fatou components. We answer this in the affirmative, using a direct construction and a recent result on the size of multiply connected Fatou components \cite[Theorem 1.2]{2011arXiv1109.1794B} to prove the following.
\begin{theo}
\label{T1}
There is a {\tef} with a simply connected fast escaping Fatou component, and no multiply connected Fatou components.
\end{theo}

To prove Theorem~\ref{T1} we require a new sufficient condition for points to be in $A(f)$, which may be of independent interest.
\begin{theo}
\label{T2}
Suppose that $f$ is a {\tef}. Suppose also that, for $R_0>0$, $\epsilon:[R_0, \infty)\to(0,1)$ is a nonincreasing function, such that
\begin{equation}
\label{T2constraint}
\epsilon(M^n(r)) \geq \epsilon(r)^{n+1}, \qfor r \geq R_0 \text{ and } n\isnatural.
\end{equation}
Define $\eta(r) = \epsilon(r) M(r)$, for $r \geq R_0$. Then there exists $R_1 \geq R_0$ such that
\begin{equation*}
A(f) = \{z : \text{ there exists } \ell \isnatural \text{ such that } |f^{n+\ell}(z)|\geq \eta^n(R'), \text{ for } n\isnatural \},
\end{equation*}
for $R' \geq R_1$.
\end{theo}
Note that this is a generalisation of \cite[Theorem 2.7]{fast}, which is obtained from Theorem~\ref{T2} when $\epsilon$ is constant. 
%
%
%
%
%
\section{The definition of the function}
\label{Sdef}
In this section we define a {\tef}, $f$, which has all the properties defined in Theorem~\ref{T1}. Since $f$ is very complicated, we first outline informally the construction of $f$, starting with simpler functions which only have some of these properties. We then give the full construction. A detailed proof of Theorem~\ref{T1} is given in subsequent sections.

Consider first a {\tef} defined by a power series;
\begin{equation*}
g(z) = z \prod_{k=1}^\infty \left(1 + \frac{z}{a_k}\right)^2, \quad 0<a_1<a_2<\cdots.
\end{equation*}
The sequence $(a_n)_{n\isnatural}$ can be chosen so that the following holds: we can define another sequence, $(b_n)_{n\isnatural}$, such that $b_n$ is approximately equal to $a_n$, $-b_n$ is close to a critical point of $g$, and $g(-b_n)$ is close to $-b_{n+1}$. It can then be shown that a small disc centred at $-b_n$ is mapped by $g$ into a small disc centred at $-b_{n+1}$. By Montel's theorem, these discs must be in the Fatou set of $g$. Moreover, these discs cannot be in multiply connected Fatou components of $g$ since, by \cite[Theorem 1.2]{2011arXiv1109.1794B}, any open set contained in a multiply connected Fatou component of $g$ must, after a finite number of iterations of $g$, cover an annulus surrounding the origin. Finally, it can be shown, by comparing $|g(-b_n)|$ to $M(b_n, g) = g(b_n)$, that these discs are contained in fast escaping Fatou components of $g$.

However, $g$ does not have all the properties defined in Theorem~\ref{T1}. In particular, by considering the behaviour of $g$ in large annuli which omit the zeros of $g$, it can be shown that $g$ has multiply connected Fatou components. Thus $g$ has very similar properties to the example in \cite{MRunknown}.

We note that no zero of $g$ can be in a multiply connected Fatou component, since 0 is a fixed point. In order to prevent the existence of multiply connected Fatou components, we add further zeros to the function, along the negative real axis. This requires some care. The addition of too many zeros -- for example, spaced linearly along the negative real axis -- leads to a breakdown of other parts of the construction. The addition of a zero with modulus insufficiently distant from $a_n$ leads to a similar breakdown.

We use \cite[Theorem 1.2]{2011arXiv1109.1794B} to show that only a relatively small number of additional zeros are required. In particular, suppose that $h$ is a {\tef} with $h(0)=0$ and with zeros of modulus $r_0 < r_1 < r_2 < \cdots$. Then, by \cite[Theorem 1.2]{2011arXiv1109.1794B}, $h$ has no multiply connected Fatou components if $\lim_{k\rightarrow\infty} \log r_{k+1}/\log r_{k}$ exists and is equal to $1$.

To use this result, we need to understand the behaviour of $\log a_{n+1} / \log a_{n}$, for large $n$. From the recursive definition that we use to ensure that the sequences $(a_n)_{n\isnatural}$ and $(b_n)_{n\isnatural}$ have the required properties, (see (\ref{thean})), we find that, for large $n$, $\log a_{n+1}/\log a_{n}$ is close to $n^3$. See (\ref{mainidentity}) for a precise statement of how the term $n^3$ arises here.

This suggests the following. Define $\mu_n = n^{3/n}$, for $n\isnatural$. To simplify some displays we set $\mu_{n,m} = \mu_n^m$, and observe that $\mu_{n,0} = 1$ and $\mu_{n,n} = n^3$, for $n\isnatural$. We now define a more complicated {\tef};
\begin{equation*}
h(z) = z \prod_{k=1}^\infty\prod_{l=0}^{k-1}\left(1+\frac{z}{a_k^{\mu_{k,l}}}\right)^2, \quad 0<a_1<a_2<\cdots.
\end{equation*}
The sequence $(a_n)_{n\isnatural}$ in this definition is not the same as in the definition of $g$, but serves the same purpose, and is chosen similarly. This function has zeros of modulus $$\cdots, a_n,~a_n^{\mu_n},~a_n^{\mu_{n,2}},~\cdots,~a_n^{\mu_{n,n-1}}, a_{n+1}, ~a_{n+1}^{\mu_{n+1}} \cdots.$$

Since it is readily seen that $\mu_n\rightarrow 1$ as $n\rightarrow\infty$, this function does not have multiply connected Fatou components. However, two further adjustments are required. Firstly, the zero of modulus $a_n^{\mu_{n,n-1}}$ is sufficiently close to the zero of modulus $a_{n+1}$ that the original construction breaks down. We resolve this by omitting this zero. Secondly, the value of $\log a_{n+1}/\log a_{n}$ is not close enough to $n^3$, for large $n$, to ensure that $\lim_{k\rightarrow\infty} \log a_{k+1}/\log a_k^{\mu_{k,k-2}} = 1.$ We resolve this by adding one additional zero, which serves no other purpose in the construction. This zero is defined using two additional sequences, $(\alpha_n)_{n\isnatural}$ and $(\beta_n)_{n\isnatural}$, which we choose to keep $\log a_{n+1}/\log a_{n}$ sufficiently close to $n^3$.

Now we are able to indicate the form of the function $f$ in Theorem~\ref{T1}. 
Let $f$ be the {\tef};
\begin{equation}
\label{thefun}
f(z) = z \prod_{k=3}^\infty \left\{ \left(1+\frac{z}{a_k^{\beta_k}}\right)^{2\alpha_{k}} \ \prod_{l=0}^{k-2}\left(1+\frac{z}{a_k^{\mu_{k,l}}}\right)^2 \right\},
\end{equation}
where $0<a_3<a_4<\cdots, \ \alpha_n\in\{0,1,2,\ldots\}, \ \beta_n\isreal,$ for $n\isnatural$. Again, the sequence $(a_n)_{n\isnatural}$ in this definition is not the same as that in the definition of $g$ or $h$, but serves the same purpose, and is chosen similarly. The related sequence $(b_n)_{n\isnatural}$, discussed after the definition of $g$, is defined for $f$ by (\ref{thebn}). The sequences $(\alpha_n)_{n\isnatural}$ and $(\beta_n)_{n\isnatural}$ are the two sequences mentioned at the end of the previous paragraph. \\

The structure of the proof of Theorem~\ref{T1} is as follows. First, in Section~\ref{S1}, we give the definition of the various sequences in (\ref{thefun}), and we prove a number of estimates on the modulus of the zeros of $f$. Next, in Section~\ref{S2}, we show that there are no multiply connected Fatou components of $f$. In Section~\ref{S3} we show that there are intervals on the negative real axis each  contained in some Fatou component of $f$. Finally, in Section~\ref{S4} we prove Theorem~\ref{T2} and then use this to show that these Fatou components of $f$ are fast escaping. It is clear that Theorem~\ref{T1} follows from these results. \\

%
%
\itshape Remark: \normalfont
Rippon and Stallard asked \cite[Question 1]{fast} if there can be unbounded fast escaping Fatou components of a {\tef}. It can be shown that the Fatou components of the function $f$ are all bounded. Indeed, it is straightforward to prove that the number of zeros of $f$ in the disc $\{z: |z| < r\}$ is $O(\log r)$, and hence that $\log M(r, f) = O((\log r)^2)$. It follows that $f$ has no unbounded Fatou components \cite[Theorem 1.9(b)]{fast}, and, moreover, that the set $A(f)$ has a structure known as a spider's web. For more details we refer to \cite{fast}.
%
%
\section{Defining the sequences}
\label{S1}
In this section we first define the sequences $(\alpha_n)_{n\isnatural}$ and $(\beta_n)_{n\isnatural}$, and then define the sequences $(a_n)_{n\isnatural}$ and $(b_n)_{n\isnatural}$. \\

Recall from Section~\ref{Sdef} that $\mu_n = n^{3/n}$ and $\mu_{n,m} = \mu_n^m$; we also define, for $n\geq 3$,
\begin{equation}
\label{sigmadef}
\sigma_n = \sum_{l=1}^{n-2} \mu_{n,l} = \mu_n \frac{\mu_{n,n-2} - 1}{\mu_n - 1}.
\end{equation}
%
%
We define $(\alpha_n)_{n\isnatural}$ to be a sequence of integers and $(\beta_n)_{n\isnatural}$ to be a sequence of real numbers. Assume that $N_0$ is even and chosen sufficiently large for subsequent estimates to hold. Set 
\begin{equation}
\label{thealphan}
 2\alpha_n =
  \begin{cases}
   0, &\text{for } n < N_0, \\
   N_0^3 + 2N_0^2 + 6N_0 + 2, &\text{for } n = N_0, \\
   3n^2 + n + 6, &\text{for } n > N_0, \ n \text{ even}, \\
   3n^2 + n + 4, &\text{for } n > N_0, \ n \text{ odd}.
  \end{cases}
\end{equation}
Note that $\alpha_n$ is an integer, for $n\isnatural$. Set
\[
 \beta_n =
  \begin{cases}
   0,                                       &\text{for } n < N_0, \\
   \frac{1}{\alpha_n}(n^4 - \sigma_n),      &\text{for } n \geq N_0, \ n \text{ even}, \\
   \frac{1}{\alpha_n}(\frac{n^3(2n-1)}{2} - \sigma_n), &\text{for } n \geq N_0, \ n \text{ odd}.
  \end{cases}
\]
We observe that these choices imply that
\begin{equation}
\label{taudef}
\tau_n = \frac{2}{n^3}\left(\alpha_{n}\beta_{n} + \sigma_n\right)
\end{equation}
satisfies
\begin{equation}
\label{thetaun}
\tau_n =
  \begin{cases}
   2n,     &\text{for } n \geq N_0, \ n \text{ even}, \\
   2n - 1, &\text{for } n \geq N_0, \ n \text{ odd},
  \end{cases}
\end{equation}
and
\begin{equation}
\label{alphaeq}
2\alpha_n = 3n^2 + n + 3 + \tau_n - \tau_{n-1},  \qfor n > N_0.
\end{equation}
%
%
We also define a sequence of integers $(T_n)_{n\isnatural}$ by
\begin{equation}
\label{theTn}
T_n =
  \begin{cases}
   n^3 + 2n - 3, &\text{for } n \text{ even}, \\
   n^3 + 2n - 2, &\text{for } n \text{ odd}.
  \end{cases}
\end{equation}
%
%
Next we prove a result which gives various relationships between these sequences.
\begin{lemma}
\label{lemx}
The following all hold for the choice of sequences above:
\begin{equation}
\label{alphabeta}
\alpha_n \sim \frac{3}{2}n^2, \quad \beta_n \sim \frac{2}{3}n^2, \quad \text{as } n\rightarrow\infty,
\end{equation}
\begin{equation}
\label{alphadef}
2 \sum_{k=3}^{n} \alpha_k  = n^3 + 2n^2 + 4n + 2 + \tau_n, \qfor n\geq N_0,
\end{equation}
\begin{equation}
\label{mainidentity}
1 + 2 \sum_{k=3}^{n-1} \alpha_{k} + \sum_{k=3}^{n-1} \sum_{l=0}^{k-2} 2 = n^3 + \tau_{n-1} = T_n, \qfor n\geq N_0,
\end{equation} 
and
\begin{equation}
\label{muthing}
\mu_{n,2} < \beta_n < \mu_{n,n-3}, \qfor\text{large } n.
\end{equation}
\end{lemma}
\begin{proof}
The first half of (\ref{alphabeta}) is immediate from (\ref{thealphan}). Now, by (\ref{thealphan}), (\ref{taudef}) and (\ref{thetaun}),
\begin{align}
\label{betasim}
\beta_n \sim \frac{2n}{3}\left(n - \frac{\sigma_n}{n^3}\right), \quad \text{ as } n\toinfty.
\end{align}
We have that
\begin{equation}
\label{logeq}
\frac{x}{2} \leq \log(1+x) \leq x, \qfor 0<x<\frac{1}{2}.
\end{equation}
Putting $x = \mu_n - 1$, we obtain
\begin{equation}
\label{muineq}
\frac{3}{n} \log n \leq \mu_n - 1 \leq \frac{6}{n} \log n, \qfor \text{large } n.
\end{equation}
Hence, by (\ref{sigmadef}),$$\frac{\sigma_n}{n^3} = \frac{\mu_n}{\mu_{n,n}}\frac{\mu_{n,n-2} - 1}{\mu_n - 1} \sim \frac{1}{\mu_n(\mu_n - 1)} = O\left(\frac{n}{\log n}\right) \quad \text{ as } n\toinfty,$$ and the second half of (\ref{alphabeta}) follows by (\ref{betasim}).

We can see that (\ref{alphadef}) holds by induction. For, it is immediately satisfied when $n=N_0$. When $n=m>N_0$ we have, by (\ref{alphaeq}) and (\ref{alphadef}) with $n=m-1$,
\begin{align*}
2\sum_{k=3}^{m} \alpha_k  &= 2\alpha_m + 2 \sum_{k=3}^{m-1} \alpha_k 
                           = m^3 + 2m^2 + 4m + 2 + \tau_m.
\end{align*}
Finally (\ref{mainidentity}) follows from (\ref{thetaun}), (\ref{theTn}), and (\ref{alphadef}), and (\ref{muthing}) follows from (\ref{alphabeta}).
\end{proof}
%
%
Next we define the sequence $(a_n)_{n\isnatural}$ recursively, and for each $n\isnatural$ put
\begin{equation}
\label{thebn}
b_n = a_n - \frac{2}{T_n+2} a_n = \frac{T_n}{T_n+2} a_n.
\end{equation}
Choose $a_3$ and $N_1$ large, and set $a_{n+1} = a_n^{n^3}$, for $3 \leq n < N_1$. We assume that $a_3$ and $N_1$ are chosen sufficiently large for various estimates in the sequel to hold. For $n \geq N_1$, we define
\begin{equation}
\label{thean}
a_{n+1} = \frac{(T_{n+1} + 2)}{T_{n+1}} \ b_n \left(1 - \frac{b_n}{a_n}\right)^2 \prod_{k=3}^{n-1} \left\{\left(1-\frac{b_n}{a_k^{\beta_k}}\right)^{2\alpha_{k}} \ \prod_{l=0}^{k-2}\left(1-\frac{b_n}{a_k^{\mu_{k,l}}}\right)^2 \right\}.
\end{equation}

%
%
Finally in this section we prove a set of inequalities which concern the growth of the sequence $(a_n)$, and the ratios of these numbers to the modulus of the other zeros of $f$. Note that (\ref{alphabeta}) and (\ref{aapprox2}) imply that the product in (\ref{thefun}) is locally uniformly convergent in $\mathbb{C}$.
\begin{lemma}
\label{lem1}
The following inequalities hold for the sequence $(a_n)$ defined above. For $n \geq 3$,
\begin{align}
\label{aapprox}
&a_n^{n^3-2/n} \leq a_{n+1} \leq a_n^{n^3+2/n}, \\
\label{aapprox2}
&a_n > \exp(e^n),
\end{align}
and, for large $n$,
\begin{align}
\label{i1}
&\frac{a_n}{a_n^{\mu_n}} \leq \exp(-e^{n/2}), \quad &\frac{\alpha_n a_n}{a_n^{\beta_n}} \leq \exp(-e^{n/2}), \\
\label{i1y}
&\frac{a_n^{\mu_{n,n-2}}}{a_{n+1}} \leq \exp(-e^{n/2}), \quad &\frac{\alpha_n a_n^{\beta_n}}{a_{n+1}} \leq \exp(-e^{n/2}).
\end{align}
\end{lemma}
%
%
\begin{proof}
First, assume that (\ref{aapprox}) holds for $3\leq n\leq m$. Equation (\ref{aapprox2}) follows for $3\leq n\leq m$ by a simple induction. Hence, for sufficiently large $m$, by (\ref{alphabeta}), (\ref{muineq}), (\ref{aapprox}) and (\ref{aapprox2}):
\begin{align*}
&\frac{a_m}{a_m^{\mu_m}} = a_m^{1-\mu_m} \leq \exp(e^m(1-\mu_m)) \leq \exp(-e^{m/2}); \\
&\frac{\alpha_m a_m}{a_m^{\beta_m}} \leq 3m^2 a_m^{1 - m^2/2} \leq \exp(-e^{m/2}); \\
&\frac{a_m^{\mu_{m,m-2}}}{a_{m+1}} \leq (a_m^{\mu_{m,m-2}})^{1 - \mu_{m,2} + \frac{2}{m}} \leq \exp(e^m(1 - \mu_{m,2} + \frac{2}{m})) \leq \exp(-e^{m/2}); \\
&\frac{\alpha_m a_m^{\beta_m}}{a_{m+1}} \leq 3m^2 a_m^{m^2 - m^3 + 2/m} \leq \exp(-e^{m/2}).
\end{align*}
It remains to prove (\ref{aapprox}). We can assume, by taking $N_1$ sufficiently large, that (\ref{aapprox}) holds for $3\leq n\leq m-1$, for some large $m$. We can assume also that $m$ is sufficiently large that (\ref{mainidentity}), (\ref{muthing}) and various other estimates used in the following hold. We need to prove that (\ref{aapprox}) holds for $n = m$. Now, by (\ref{thean}) and (\ref{taudef}),
\begin{align}
a_{m+1} &= \frac{(T_{m+1} + 2)}{T_{m+1}} \ b_m^{\kappa_m} \left(1 - \frac{b_m}{a_m}\right)^2 \frac{L_1}{\prod_{k=3}^{m-1} \left\{a_k^{2\alpha_k\beta_k} \prod_{l=0}^{k-2} a_k^{2 \mu_{k,l}} \right\}} \nonumber \\
        &= k_m \frac{a_m^{\kappa_m}}{a_{m-1}^{2 \alpha_{m-1}\beta_{m-1}} \prod_{l=1}^{m-3} a_{m-1}^{2 \mu_{m-1,l}}} \frac{L_1}{L_2} \nonumber \\
\label{theamthing}
        &= k_m \frac{a_m^{\kappa_m}}{a_{m-1}^{(m-1)^3\tau_{m-1}}} \frac{L_1}{L_2},
\end{align}
where, by (\ref{thebn}),
\begin{equation*}
k_m = \frac{(T_{m+1} + 2)}{T_{m+1}} \left(\frac{T_m}{T_m+ 2}\right)^{\kappa_m} \left(\frac{2}{T_m+2}\right)^2,
\end{equation*}
by (\ref{mainidentity}), 
\begin{align*}
\kappa_m = 1 + 2 \sum_{k=3}^{m-1} \alpha_k + \sum_{k=3}^{m-1} \sum_{l=0}^{k-2} 2 = m^3 + \tau_{m-1} = T_m,
\end{align*}
\begin{equation*}
L_1 = \prod_{k=3}^{m-1} \left\{\left(1 - \frac{a_k^{\beta_k}}{b_m}\right)^{2\alpha_k}  \prod_{l=0}^{k-2}\left(1-\frac{a_k^{\mu_{k,l}}}{b_m}\right)^2 \right\},
\end{equation*}
and, by (\ref{taudef}) again,
\begin{equation*}
L_2 = a_{m-1}^2 \prod_{k=3}^{m-2} \left\{a_k^{2 \alpha_k\beta_k} \prod_{l=0}^{k-2} a_k^{2 \mu_{k,l}}\right\} = a_{m-1}^2 \prod_{k=3}^{m-2} a_k^{2(1+\alpha_k\beta_k+\sum_{l=1}^{k-2} \mu_{k,l})} = a_{m-1}^2 \prod_{k=3}^{m-2} a_k^{2+k^3\tau_k}.
\end{equation*}
We now estimate the terms in this equality. Firstly, by (\ref{theTn}), and noting that $\left(T_m/(T_m+2)\right)^{T_m} > 1/e^2$, we obtain $$\frac{1}{8}m^{-6} < k_m < 8m^{-6}.$$

Secondly, by (\ref{aapprox}), with $n=m-1$, $$a_{m-1}^{(m-1)^3} \leq a_m^{(1 - \frac{2}{(m-1)^4})^{-1}} \leq a_m^{1+\frac{4}{(m-1)^4}},$$ and so, by (\ref{thetaun}), $$\frac{a_m^{\kappa_m}}{a_{m-1}^{(m-1)^3\tau_{m-1}}} \geq a_m^{\kappa_m - \tau_{m-1} - \frac{4\tau_{m-1}}{(m-1)^4}} = a_m^{m^3 - \frac{4\tau_{m-1}}{(m-1)^4}} \geq a_m^{m^3 - \frac{1}{m}}.$$ Similarly, by (\ref{thetaun}) and (\ref{aapprox}), $$\frac{a_m^{\kappa_m}}{a_{m-1}^{(m-1)^3\tau_{m-1}}} \leq a_m^{m^3 + \frac{1}{m}}.$$

Thirdly, we consider $L_1$. Noting (\ref{muthing}), and by (\ref{thebn}) and (\ref{i1y}), we see that the smallest term in this product is $$1 > \left(1-\frac{a_{m-1}^{\mu_{m-1,m-3}}}{b_m}\right)^2 > \left(1-\frac{2a_{m-1}^{\mu_{m-1,m-3}}}{a_m}\right)^2 > 1 - \exp(-e^{\frac{m}{4}}).$$ There are fewer than $m$ terms in $L_1$ of the form $(1 - a_{m-1}^p/b_m)^q, \ p \isreal, \ q \isnatural$, and so of comparable size to this term. The other terms in the product for $L_1$ tend to $1$ sufficiently quickly, by (\ref{aapprox}), that $1/2 < L_1 < 1$.

Finally, we consider $L_2$. Observe that, by (\ref{thetaun}) and (\ref{aapprox}), the largest term in this product is $$a_{m-1}^2 a_{m-2}^{2+(m-2)^3\tau_{m-2}} < a_{m-1}^2 a_{m-2}^{4(m-2)^4} <  a_{m-1}^2 a_{m-1}^{8(m-1)} < a_m^{16/m^2}.$$ By (\ref{aapprox}), all other terms in the product for $L_2$ decrease sufficiently quickly that $1~<~L_2~<~a_m^{32/m^2}$. \\

Thus, by (\ref{theamthing}), for sufficiently large $m$,$$a_{m+1} \geq \frac{1}{16}m^{-6} a_m^{m^3 - \frac{1}{m}-\frac{32}{m^2}} \geq a_m^{m^3 - \frac{2}{m}},$$ and similarly, $$a_{m+1} \leq 8m^{-6} a_m^{m^3 + \frac{1}{m}} \leq a_m^{m^3 + \frac{2}{m}}.$$ 
This completes the proof of Lemma~\ref{lem1}.
\end{proof}
%
%
\section{There are no multiply connected Fatou components}
\label{S2}
In this section we prove the following result.
\begin{lemma}
\label{lem2}
The {\tef} $f$ does not have multiply connected Fatou components.
\end{lemma}
We use the definition of an annulus $$A(r_1, r_2) = \{ z : r_1 < |z| < r_2 \}, \qfor 0 < r_1 < r_2.$$ We need the following, which is part of \cite[Theorem 1.2]{2011arXiv1109.1794B}.
\begin{thma}
\label{Ta}
Suppose that $g$ is a {\tef} with a multiply connected Fatou component $U$. For each $z_0 \in U$ there exists $\alpha > 0$ such that, for sufficiently large $n \isnatural$, $$g^n(U) \supset A(|g^n(z_0)|^{1-\alpha}, |g^n(z_0)|^{1+\alpha}).$$
\end{thma}
\begin{proof}[Proof of Lemma~\ref{lem2}]
Observe that, for large $n$, in the closed annulus $\overline{A(a_n, a_{n+1})}$ there are zeros of $f$ on the negative real axis of modulus $a_n, a_n^{\mu_n}, a_n^{\mu_{n,2}}, \ldots, a_n^{\mu_{n,n-2}}$ and $a_{n+1}$. Note also that 0 is a fixed point of $f$, and so no zero of $f$ can be in a multiply connected Fatou component of $f$. Now, by (\ref{aapprox}),
\begin{equation*}
a_{n+1} \leq a_n^{n^3 + 2/n} < (a_n^{\mu_{n,n-2}})^{\mu_{n,2}+2/n}.
\end{equation*}
Hence, for large $n$, there is at least one zero of $f$ in any annulus $A(r, r^{\mu_{n,2}+2/n})$, for $a_n \leq r \leq a_{n+1}$. Note that $\mu_{n,2}+2/n \rightarrow 1$ as $n\rightarrow\infty$.

Now, by Theorem~A, if $f$ has a multiply connected Fatou component, then there is a $c>1$, and a sequence $(r_i)_{i\isnatural}$, tending to infinity, such that the annuli $A(r_i, r_i^c)$ are contained in multiply connected Fatou components of $f$.  This is in contradiction to the observations above regarding the distribution of zeros of $f$ and the fact that these zeros do not lie in multiply connected Fatou components. Hence there can be no multiply connected Fatou components of $f$.
\end{proof}
%
%
%
%
\section{There are simply connected Fatou components}
\label{S3}
Next we show that $f$ has simply connected Fatou components.
\begin{lemma}
\label{lem3}
Define $B_n = \{z : |z + b_n| < \delta_n b_n\},$ where $\delta_n = n^{-15}$. Then, for large $n$, we have $f(B_n) \subset B_{n+1}$, and $B_n$ is contained in a simply connected Fatou component of $f$.
\end{lemma}
%
%
\begin{proof}
Suppose that $z \in B_n$, in which case $z = -b_n + w b_n$ where $|w| < \delta_n$. We assume throughout this section that $n$ is sufficiently large for various estimates to hold. By (\ref{thefun}),  (\ref{thebn}) and (\ref{thean}), $$\frac{f(z)}{-b_{n+1}} = I_1 I_2,$$ where
\begin{align}
\label{I1def}
I_1 &= (1-w) \left(1 + w \frac{b_n}{a_n-b_n}\right)^2 \prod_{k=3}^{n-1} \left\{\left(1+w\frac{b_n}{a_k^{\beta_k}-b_n}\right)^{2\alpha_{k}} \prod_{l=0}^{k-2}\left(1+w\frac{b_n}{a_k^{\mu_{k,l}}-b_n}\right)^{2} \right\},
\end{align}
and
\begin{equation*}
I_2 = \prod_{l=1}^{n-2} \left(1 + \frac{z}{a_n^{\mu_{n,l}}}\right)^{2}\ \prod_{k=n}^\infty \left(1 + \frac{z}{a_k^{\beta_k}}\right)^{2\alpha_{k}} \prod_{k=n+1}^\infty \prod_{l=0}^{k-2} \left(1 + \frac{z}{a_k^{\mu_{k,l}}}\right)^{2}.
\end{equation*}
First consider $I_1$. This is a polynomial in $w$ of degree $T_n + 2$, by (\ref{mainidentity}). Write $$I_1 = 1 + \sum_{j=1}^{T_n+2} \eta_j w^j, \qfor \eta_j\iscomplex.$$ Then
\begin{equation}
\label{i1ineq}
|I_1 - 1| \leq |\eta_1 w| + |\eta_2 w^2| + \cdots + |\eta_{T_n+2} w^{T_n+2}|.
\end{equation}
We consider $\eta_1$. We have, by (\ref{thealphan}), (\ref{mainidentity}), (\ref{thebn}), (\ref{aapprox}), (\ref{aapprox2}) and (\ref{i1y}),
\begin{align}
|\eta_1| &= \left|1 + 2\frac{b_n}{b_n-a_n} + 2\sum_{k=3}^{n-1} \left\{\alpha_{k} \frac{b_n}{b_n-a_k^{\beta_k}} + \sum_{l=0}^{k-2} \frac{b_n}{b_n-a_k^{\mu_{k,l}}} \right\}\right| \nonumber \\
\label{simpleq}
           &= \left |1 - T_n + 2\sum_{k=3}^{n-1}\left\{\alpha_{k}\left(1 + \frac{a_k^{\beta_k}}{b_n-a_k^{\beta_k}}\right) + \sum_{l=0}^{k-2} \left(1 + \frac{a_k^{\mu_{k,l}}}{b_n-a_k^{\mu_{k,l}}}\right) \right\}\right| \\
           &= 2\left |\sum_{k=3}^{n-1} \left\{\alpha_{k}\frac{a_k^{\beta_k}}{b_n-a_k^{\beta_k}} + \sum_{l=0}^{k-2} \frac{a_k^{\mu_{k,l}}}{b_n-a_k^{\mu_{k,l}}}\right\}\right|  \nonumber \\
           &\leq 4\left|n\alpha_{n-1}\frac{a_{n-1}^{\beta_{n-1}}}{a_n} + n^2 \frac{a_{n-1}^{\mu_{n-1,n-3}}}{a_n} \right|  \nonumber \\
           &\leq \exp(-e^{n/4}). \nonumber
\end{align}    
Note that the cancellation in (\ref{simpleq}) occurs because, due to the choice of $b_n$ and $T_n$, $-b_n$ is very close to a critical point of $f$.
       
Next consider $\eta_k$, for $k>1$. Note that the coefficients of $w$ in the product for $I_1$ have modulus at most $b_n/(a_n - b_n) = T_n/2 < n^3$. Moreover, the degree of $I_1$ is less than $2n^3$, and so the expansion of (\ref{I1def}) contains less than $(2n^3)^k$ terms in $w^k$. Hence $$|\eta_k| < (n^3)^k.(2n^3)^k < n^{7k},\qfor\text{large }n.$$ Hence $|\eta_2 w^2| < n^{-16}$. The other terms in (\ref{i1ineq}) decrease sufficiently quickly that 
\begin{equation}
\label{i1eq}
|I_1 - 1| < 2n^{-16}, \qfor \text{large } n.
\end{equation}
Now we consider $I_2$. Observe that each term in the product for $I_2$ has modulus less than 1. Hence, since $-\log(1-x) \leq \log (1+2x)$, for $0<x<1/2$, we have, by (\ref{logeq}), (\ref{aapprox}), (\ref{aapprox2}) and (\ref{i1}),
\begin{align*}
0 < -\log |I_2| &\leq 2\left(\sum_{l=1}^{n-2} \log \left(1+\frac{2|z|}{a_n^{\mu_{n,l}}}\right) + \sum_{k=n}^\infty \alpha_{k} \log \left(1 + \frac{2|z|}{a_k^{\beta_k}}\right)+ \sum_{k=n+1}^\infty \sum_{l=0}^{k-2} \log \left(1 + \frac{2|z|}{a_k^{\mu_{k,l}}}\right)\right) \\
             &\leq 8\left(\sum_{l=1}^{n-2} \frac{a_n}{a_n^{\mu_{n,l}}} + \sum_{k=n}^\infty \alpha_{k} \frac{a_n}{a_k^{\beta_k}} + \sum_{k=n+1}^\infty \sum_{l=0}^{k-2} \frac{a_n}{a_k^{\mu_{k,l}}} \right) \\
             &\leq 16\left(\exp(-e^{n/2}) + \exp(-e^{n/2}) + \exp(-e^{n/2})\right) \\
             &\leq \exp(-e^{n/4}).
\end{align*}
Thus, $1 - 2\exp(-e^{n/4}) \leq |I_2| \leq 1.$ This, together with (\ref{i1eq}), establishes the first part of the lemma. It follows from Montel's theorem that, for large $n$, $B_n$ is contained in a Fatou component, which must be simply connected by Lemma~\ref{lem2}.
\end{proof}
\itshape Remark. \normalfont Let $V_n$ be the Fatou component containing $B_n$. These Fatou components are distinct. For, suppose that $V_m = V_n$ with $m \ne n$. Because all the coefficients of $z$ in (\ref{thefun}) are real, the Fatou set $F(f)$ must be invariant under reflection in the real axis. Hence, all points on the negative real axis between $B_n$ and $B_m$ must be in $V_m$, as otherwise $V_m$ would be multiply connected. This is a contradiction since these points include the zeros of $f$.
%
%
\section{The simply connected Fatou components are fast escaping}
\label{S4}
In this section we first prove Theorem~\ref{T2}, and then we use this result to prove the following.
\begin{lemma}
\label{lem5}
Let $V_n$, $n\isnatural$, be the simply connected Fatou components defined at the end of Section~\ref{S3}. Then $V_n \subset A(f)$, for large $n$.
\end{lemma}
%
%
\begin{proof}[Proof of Theorem~\ref{T2}]
%
%
We need two facts about the maximum modulus function, $M(\rho)$, of a {\tef}. The first is well known, and the second is from \cite[Lemma 2.2]{MR2544754}:
\begin{equation}
\label{M2}
\frac{\log M(\rho)}{\log \rho} \rightarrow\infty \text{ as } \rho\rightarrow\infty,
\end{equation}
and there exists $R > 0$ such that
\begin{equation}
\label{M4}
M(\rho^c) \geq M(\rho)^c, \qfor \rho \geq R, \ c>1.
\end{equation}
%
%
Fix $r_0 \geq R_0$ such that $M(r) > r$, for $r\geq r_0$. Whenever $r \geq r_0$ there is a unique $n\isnatural$ such that $M^{n-1}(r_0) \leq r < M^n(r_0)$. Hence, since $\epsilon$ is nonincreasing, by (\ref{T2constraint}) and  (\ref{M2})
\begin{equation*}
\epsilon(r) r \geq \epsilon(M^n(r_0))M^{n-1}(r_0) \geq \epsilon(r_0)^{n+1}M^{n-1}(r_0)\toinfty \text{ as } n\toinfty.
\end{equation*}
Hence
\begin{equation}
\label{e2}
\epsilon(r) r \toinfty \text{ as } r \toinfty.
\end{equation}
By (\ref{M2}) and (\ref{e2}) we see that, given $k~>~0$, we can ensure that
\begin{equation}
\label{einf}
\frac{\log M(\epsilon(r) R)}{\log (\epsilon(r) R)} > k, \qfor \text{large } r, \ R \geq r.
\end{equation}
A little algebra shows that this is equivalent to
\begin{equation}
\label{e4}
M(\epsilon(r) R)^{\frac{-\log\epsilon(r)}{\log (\epsilon(r) R)}} > \epsilon(r)^{-k}, \qfor \text{large } r, \ R \geq r.
\end{equation}
In (\ref{M4}) we replace $c$ with $\log R / \log(\epsilon(r) R)$, and replace $\rho$ with $\epsilon(r) R$. We obtain, using (\ref{e4}) with $k=3$, that there exists $R_1 \geq R_0$ such that
\begin{equation}
\label{Ae1}
M(R) \geq \epsilon(r)^{-3} M(\epsilon(r) R), \qfor R \geq r \geq R_1.
\end{equation}
Note that we can assume that $R_1$ is sufficiently large that $M(r) > r$, for $r \geq R_1/\epsilon(r)$, and also, by  (\ref{e2}), that $\eta(r) > r$, for $r \geq R_1$. \\

We claim next that we have
\begin{equation}
\label{eind}
\eta^k(r) \geq \epsilon(r)^{-k-1} M^k(\epsilon(r)r) > M^k(\epsilon(r)r), \qfor r \geq R_1, \ k\isnatural.
\end{equation}
This can be seen by induction. When $k = 1$ we have, by (\ref{Ae1}),
\begin{equation*}
\eta(r) = \epsilon(r) M(r) \geq \epsilon(r)^{-2} M(\epsilon(r) r), \qfor r \geq R_1.
\end{equation*} 
Hence, by induction, for $r \geq R_1$,
\begin{align*}
\eta^{k+1}(r) 
              &=    \epsilon(\eta^k(r)) M(\eta^k(r)) \\        
              &\geq \epsilon(M^k(r)) M(\eta^k(r)) \ &\text{as } \epsilon \text{ is nonincreasing} \\             
							&\geq \epsilon(r)^{k+1} M(\eta^k(r))     \ &\text{by } (\ref{T2constraint}) \\
              &\geq \epsilon(r)^{k+1} M(\epsilon(r)^{-k-1} M^k(\epsilon(r)r))      \ &\text{by } (\ref{eind}) \\             					&\geq \epsilon(r)^{k+1} \epsilon(r)^{-3k-3} M^{k+1}(\epsilon(r)r)   \ &\text{by repeated use of } (\ref{Ae1}) \\
							&\geq \epsilon(r)^{-(k+1)-1} M^{k+1}(\epsilon(r)r)									\ &\text{as required.}						
\end{align*}
(Note that in the penultimate step we have also used the fact that $\epsilon(r) r \leq M(\epsilon(r)r)$, for $r\geq R_1$.)

It follows from (\ref{eind}) that, for $r \geq R_1$, $\eta^n(r)\toinfty$ as $n\toinfty$. Define
\begin{equation*}
A'(f) = \{z : \text{ there exists } \ell \isnatural \text{ such that } |f^{n+\ell}(z)|\geq \eta^n(R') \text{ for } n\isnatural \},
\end{equation*}
for $R' \geq R_1$. We complete the proof by showing that $A'(f) = A(f)$. 

First, suppose that $z \in A(f)$, in which case for some $\ell\isnatural$ we have $$|f^{n+\ell}(z)| \geq M^n(R), \qfor n\isnatural,$$ and some $R$ with $M(r) > r$, for $r \geq R$. Choose $K\isnatural$ such that $M^K(R)~=~R'~\geq~R_1$. Then
\begin{equation*}
|f^{n+\ell+K}(z)| \geq M^{n+K}(R) = M^n(R') \geq \eta^n(R'), \qfor n\isnatural.
\end{equation*}
Hence $z \in A'(f)$.

Conversely, suppose that $z \in A'(f)$, in which case for some $\ell\isnatural$ and $R'\geq R_1$ we have $$|f^{n+\ell}(z)| \geq \eta^n(R'), \qfor n\isnatural.$$ Choose $K\isnatural$ so that $M^K(\epsilon(R')R') = R \geq R_1$. Then, by (\ref{eind}).
\begin{equation*}
|f^{n+\ell+K}(z)| \geq \eta^{n+K}(R') \geq M^{n+K}(\epsilon(R') R') \geq M^n(R), \qfor n\isnatural.
\end{equation*}
Hence $z \in A(f)$.
\end{proof}
Finally, we give the 
%
%
\begin{proof}[Proof of Lemma~\ref{lem5}]
For some large $R_0$ define, for $r > R_0$,
\begin{equation}
\epsilon(r) = \frac{1}{16n^6}, \qfor a_{n-1}(1 + \delta_{n-1}) < r \leq a_n(1 + \delta_n),
\end{equation}
where $\delta_n = n^{-15}$ as in Lemma~\ref{lem3}. Define also $\eta(r) = \epsilon(r) M(r)$, for $r \geq R_0$. \\

Suppose that $x' \in B_n \cap \mathbb{R} \subset V_n$, for some $n$. We can assume that $n$ is chosen sufficiently large for various estimates in this section to hold. We claim that $x'\in A(f)$. Thus, by \cite[Theorem 1.2]{fast}, $V_n \subset A(f)$. \\

Our approach to proving this claim is as follows. Set $x = -x'$, recalling that $x > 0$. We first show that
\begin{equation}
\label{isinAeq}
|f(x')| \geq \frac{1}{16n^6} M(x).
\end{equation}
It follows from this, and since $f(x') \in B_{n+1} \cap \mathbb{R}$, that 
\begin{equation}
\label{isinA}
|f^m(x')| \geq \eta^m(x), \qfor m\isnatural.
\end{equation}
Second, we show that $\epsilon$ satisfies (\ref{T2constraint}). Thus, by (\ref{isinA}) and Theorem~\ref{T2}, $x'\in A(f)$, as required. \\

First we need to establish (\ref{isinAeq}). Since $M(x) = f(x)$, we have
\begin{align*}
\frac{|f(x')|}{M(x)} = J_1 J_2 J_3,
\end{align*}
where
\begin{align*}
J_1 &= \prod_{k=3}^{n-1} \left\{\left(\frac{a_k^{\beta_k} - x}{a_k^{\beta_k} + x}\right)^{2\alpha_k} \ \sum_{l=0}^{k-2} \left(\frac{a_k^{\mu_{k,l}} - x}{a_k^{\mu_{k,l}} + x}\right)^2 \right\}, \\
J_2 &= \left(\frac{a_n - x}{a_n + x}\right)^2, \\
\text{and } \\
J_3 &= \prod_{k=n+1}^\infty \left\{\left(\frac{a_k^{\beta_k} - x}{a_k^{\beta_k} + x}\right)^{2\alpha_k} \ \sum_{l=0}^{k-2} \left(\frac{a_k^{\mu_{k,l}} - x}{a_k^{\mu_{k,l}} + x}\right)^2 \right\} \left(\frac{a_n^{\beta_n} - x}{a_n^{\beta_n} + x}\right)^{2\alpha_n} \sum_{l=1}^{n-2} \left(\frac{a_n^{\mu_{n,l}} - x}{a_n^{\mu_{n,l}} + x}\right)^2.
\end{align*}

We consider these three terms separately. By (\ref{muthing}) and (\ref{i1y}), the smallest term in $J_1$ is
\begin{equation*}
\left(\frac{a_{n-1}^{\mu_{n-1,n-3}} - x}{a_{n-1}^{\mu_{n-1,n-3}} + x}\right)^2  = \left(\frac{1 - \frac{a_{n-1}^{\mu_{n-1,n-3}}}{x}}{1 + \frac{a_{n-1}^{\mu_{n-1,n-3}}}{x}}\right)^2  \geq \left(1 - \frac{8a_{n-1}^{\mu_{n-1,n-3}}}{a_n}\right) \geq \left(1-\exp(-e^{n/4})\right).
\end{equation*}
Thus, recalling the estimates of Lemma~\ref{lem1}, $J_1 \geq \frac{1}{2}$. \\

Secondly, recalling that $x = b_n + \omega b_n$, with $|\omega| < \delta_n = n^{-15}$, we have for large $n$,
\begin{align*}
J_2 &= \left(\frac{a_n - b_n - \omega b_n}{a_n + b_n + \omega b_n}\right)^2 = \left(\frac{2 - \omega T_n}{2T_n + 2 + \omega T_n}\right)^2 
\geq \frac{1}{4n^6}.
\end{align*}

Thirdly, by (\ref{muthing}) and (\ref{i1}), the smallest term in $J_3$ is
\begin{equation*}
\left(\frac{a_{n}^{\mu_n} - x}{a_{n}^{\mu_n} + x}\right)^2 \geq \left(\frac{1 - \frac{x}{a_{n}^{\mu_n}}}{1 + \frac{x}{a_{n}^{\mu_n}}}\right)^2 \geq \left(1 - \frac{8a_n}{a_n^{\mu_n}}\right) \geq \left(1-\exp(-e^{n/4})\right).
\end{equation*}
Thus, recalling again the estimates of Lemma~\ref{lem1}, $J_3 \geq \frac{1}{2}$. This establishes our first claim. \\

Finally, we need to show that $\epsilon$ satisfies (\ref{T2constraint}). We claim that $M(a_n(1+\delta_n)) \leq a_{n+2}$. In a very similar way to the proof of (\ref{isinAeq}), we can split $f(2b_n)/a_{n+1}$ into three terms, and show that $f(2b_n)/a_{n+1} < 2^{12n^3}$. Hence, by (\ref{thebn}), (\ref{aapprox}) and (\ref{aapprox2}),
\begin{equation}
\label{Mineq}
M(a_n(1+\delta_n)) \leq M(2b_n) = f(2b_n) < 2^{12n^3} a_{n+1} \leq a_{n+1}^2 \leq a_{n+2},
\end{equation}
as required.

Suppose then that $r$ is such that $a_{n-1}(1 + \delta_{n-1}) < r \leq a_n(1 + \delta_n)$. Since $\epsilon$ is nonincreasing, we deduce that, for $k\isnatural$,
\begin{align*}
\epsilon(M^k(r)) &\geq \epsilon(M^k(a_n(1 + \delta_n))), \\
                 &\geq \epsilon(a_{n+2k}),               &\text{by (\ref{Mineq})}\\
                 &= \frac{1}{16(n+2k)^6} \geq \frac{1}{(16n^6)^{k+1}} = \epsilon(r)^{k+1}.
\end{align*}
Thus $\epsilon$ satisfies (\ref{T2constraint}). This completes the proof of Lemma~\ref{lem5} and hence the proof of Theorem~\ref{T1}. 
\end{proof}
\ack{The author is grateful to Gwyneth Stallard and Phil Rippon for their help with this paper.}
%
%
\bibliographystyle{acm}
\bibliography{Main}
\end{document}